\documentclass[12pt]{article}
\usepackage{amscd,amssymb,stmaryrd}
\usepackage[fleqn]{amsmath}
\usepackage{graphicx}
\usepackage{caption}
\usepackage{subfigure, fancybox}
\usepackage[utf8]{inputenc}
\usepackage[export]{adjustbox}
\usepackage{wrapfig}
\usepackage{amstext}
\usepackage{pstricks,pst-node,pst-plot,pst-coil}
\usepackage{amsthm}
\usepackage{amsmath}
\usepackage{color}
\usepackage{mathtools}
\usepackage{hyperref}
\usepackage[english]{babel}
\usepackage{booktabs}
\usepackage{mathrsfs}
\usepackage{comment}

\newcommand\norm[1]{\|#1\|}

\newcommand\abs[1]{\lvert#1\rvert}

\newcommand{\tforall}{\text{ for all }}

\newcommand{\divg}[1]{\nabla \cdot #1}
\newcommand{\pp}[1]{\sigma#1}
\theoremstyle{definition}
\newtheorem{theorem}{Theorem}[section]
\newtheorem{proposition}{Proposition}[section]

\newtheorem{lemma}[theorem]{Lemma}

\newtheorem{example}[theorem]{Example}
\theoremstyle{remark}
\newtheorem*{remark}{Remark}

\providecommand{\keywords}[1]{{\small{\bf{Keywords~}} #1}}
\title{Computational multiscale methods for first-order wave equation using mixed CEM-GMsFEM}
\author{Eric Chung\thanks{Department of Mathematics, The Chinese University of Hong Kong, Shatin, Hong Kong}, \and Sai-Mang Pun\thanks{Department of Mathematics, The Chinese University of Hong Kong, Shatin, Hong Kong}}
\date{ }
\begin{document}
\maketitle
\begin{abstract}
In this paper, we consider a pressure-velocity formulation of the heterogeneous wave equation and employ the constraint energy minimizing generalized multiscale finite element method (CEM-GMsFEM) to solve this problem. The proposed method provides a flexible framework to construct crucial multiscale basis functions for approximating the pressure and velocity. These basis functions are constructed by solving a class of local auxiliary optimization problems over the eigenspaces that contain local information on the heterogeneity. 
Techniques of oversampling are adapted to enhance the computational performance. 
The first-order convergence of the proposed method is proved and illustrated by several numerical tests.
\end{abstract}



\section{Introduction}\label{sec:intro}
Wave propagation and its numerical simulations have been widely studied for years due to its fundamental importance in engineering applications. 
For example, these problems arise in the study of seismic wave propagation from geoscience \cite{sato2012seismic}. 
In such applications, the background materials in the domain are often highly heterogeneous, and their elastic properties may vary with the depth rapidly. 
In those cases with non-smooth heterogeneous coefficients, direct simulation using standard numerical methods (e.g. finite element method \cite{grote2009optimal}) may lead to prohibitively expensive computational cost to resolve the heterogeneous structure of the media. However, traditional methods capture fine-scale features with moderately high computational resources \cite{delprat2008wave}. 
Therefore, it is necessary to apply model reduction techniques to alleviate the computational burden in the accurate simulations of wave propagation. 


Many model reduction techniques have been well developed in the existing literature. For example, in numerical upscaling methods \cite{gibson2014multiscale,owhadi2008numerical,vdovina2009two,vdovina2005operator}, one typically derives some {\it upscaling} media and solves the resulting upscaled problem globally on a coarse grid. The dimensions of the corresponding linear systems are much smaller, giving a guaranteed saving of computational cost. 
In addition, various multiscale methods \cite{abdulle2017multiscale,engquist2009multi,engquist2012multiscale} for simulating wave propagation are presented in the literature. For instance, multiscale finite element methods (MsFEM) \cite{jiang2012priori,jiang2010analysis} and the heterogeneous multiscale methods (HMM) \cite{abdulle2012heterogeneous,abdulle2011finite,abdulle2014finite} are proposed to discretize the wave equation in a coarse grid. Recently, a class of generalized finite element methods for the wave equation \cite{abdulle2017localized, maier2019explicit} has been proposed. This type of methods is based on the idea of localized orthogonal decomposition (LOD) \cite{Peterseim2014} and generalize the traditional finite element method to accurately resolve the multiscale problems with a cheaper cost. 


In this research, we focus on the recently developed generalized multiscale finite element method (GMsFEM) \cite{chung2016adaptive,efendiev2013generalized}. The GMsFEM is a generalization of the classical MsFEM \cite{efendiev2009multiscale} in the sense that multiple basis functions can be systematically constructed for each coarse block. The GMsFEM consists of two stages: the offline and online stages. In the offline stage, a set of (local supported) snapshot functions are constructed, which can be used to essentially capture all fine-scale features of the solution. Then, a model reduction is performed by the use of a well-designed local spectral decomposition, and the dominant modes are chosen to be the multiscale basis functions. All these computations are done before the actual simulations of the model. In the online stage, with a given source term and boundary conditions, the multiscale basis functions obtained in the offline stage are used to approximate the solution. There are some previous works using GMsFEM for the wave equation based on the second-order formulation of wave equation \cite{chung2014generalized, gao2015generalized} and the wave equation in mixed formulation \cite{chung2016mixed}. 

The objective of this work is to develop for the first-order wave equation \cite{glowinski2004solution} a new computational multiscale method based on the idea of constraint energy minimization (CEM) proposed in \cite{chung2018cemmixed}. In order to derive an energy-conserving numerical scheme for the wave equation, we consider a pressure-velocity formulation. For spatial discretization, we adopt the idea of CEM-GMsFEM presented in \cite{chung2018constraint,chung2018cemmixed} and propose a multiscale method for heterogeneous wave propagation and construct multiscale spaces for both, the velocity and the pressure variables. In this research, we show the first-order convergence of the method using CEM-GMsFEM combined with the leapfrog scheme. Numerical results are provided to demonstrate the efficiency of the proposed method. The present CEM-GMsFEM setting allows flexibly adding additional basis functions based on spectral properties of the differential operators. This enhances the accuracy of the method in the presence of high contrast in the media. It is shown that if enough basis functions are selected, the convergence of the method can be shown independently of the contrast. Unfortunately, a high number of basis functions directly influences the computational complexity of the method. The direct influence of the contrast on the needed number of basis functions is not known, but numerical results indicate that a moderate number of basis functions, depending logarithmically on the contrast, seems sufficient.

The remainder of the paper is organized as follows. We provide in Section \ref{sec:prelim} the background knowledge of the problem. Next, we introduce the multiscale method and the discretization in Section \ref{sec:method}. In Section \ref{sec:analysis}, we provide the stability estimate of the method and prove the convergence of the proposed method. 
We present the numerical results in Section \ref{sec:numerics}. Finally, we give some concluding remarks in Section \ref{sec:conclusion}. 
\section{Preliminaries}\label{sec:prelim}
Consider the wave equation in mixed formulation over the (bounded) computational domain $\Omega \subset \mathbb{R}^2$
\begin{eqnarray}
	\begin{split}
	\kappa^{-1} \dot{v} + \nabla p  = & 0  &\quad \text{in } \Omega \times (0,T], \\ 
	\rho \dot{p} + \divg{v}  = & f  &\quad \text{in } \Omega \times (0,T], \\
	v \cdot \mathbf{n}  = & 0 &\quad \text{on } \partial \Omega \times [0,T], \\
	v|_{t=0}  = & h_v &\quad \text{in } \Omega, \\
	p|_{t=0}  = & h_p &\quad \text{in } \Omega. 
	\end{split}
	\label{eqn:model}
\end{eqnarray}
Here, $\dot{v}$ and $\dot{p}$ represent the time derivatives of $v$ and $p$ respectively, $T>0$ is a given terminal time, $\rho \in L^\infty (\Omega)$ is the (positive) density of the fluid satisfying $0< \rho_{\min} \leq \rho$, and $\mathbf{n}$ is the unit outward normal vector to the boundary $\partial \Omega$. We assume that the permeability field $\kappa: \Omega \to \mathbb{R}$ is highly oscillatory, satisfying $0 < \kappa_{\min} \leq \kappa(x) \leq \kappa_{\max}$ for almost every $x \in \Omega$ with $\frac{\kappa_{\max}}{\kappa_{\min}} \gg 1$.  
The source function satisfies $f \in L^2(\Omega)$. Here, $h_v$ and $h_p$ are some given initial conditions. In general, we refer to the solution $v$ as velocity and $p$ as pressure.  We denote 
$$V_0:= \{ v \in H(\text{div}, \Omega) : v \cdot \mathbf{n} = 0 ~ \text{on } \partial \Omega \} \quad \text{and} \quad Q := L^2(\Omega).$$ 

Instead of the original PDE formulation, we consider the variational formulation corresponding to \eqref{eqn:model}: find $v \in V_0$ and $p\in Q$ such that 
\begin{eqnarray}
	a(\dot{v},w) - b(w,p) = & 0 & \quad \forall w \in V_0, \\
	(\dot{p},q)_{\rho} + b(v,q)  = & (f,q) &\quad \forall q \in Q,
\end{eqnarray}

where $(\cdot,\cdot)$ denotes the inner product in $L^2(\Omega)$ and $(\cdot,\cdot)_{\rho}$ denotes the inner product in $L^2(\Omega)$ with weighted function $\rho$. The bilinear forms $a:V_0 \times V_0 \to \mathbb{R} $ and $b: V_0 \times Q \to \mathbb{R}$ are defined as follows: 
$$ a(v,w) := \int_\Omega \kappa^{-1} v \cdot w ~dx, \quad b(v,p) := \int_\Omega p~\divg{v} ~dx, $$
for all $v, w \in V_0$, and $p\in Q$. 
We remark that the following inf-sup condition should satisfy: for all $q \in Q$ with $\int_\Omega q ~dx = 0$, there exists a constant $C_0 >0$ independent to $\kappa$ such that 
$$\norm{q}_{L^2(\Omega)} \leq C_0 \sup_{v \in V_0} \frac{b(v,q)}{\norm{v}_{H(\text{div}, \Omega)}}.$$

In this research, we will apply the constraint energy minimizing generalized multiscale finite element method (CEM-GMsFEM) for mixed formulation, which is originally proposed in \cite{chung2018cemmixed}, to approximate the solution of the above mixed problem. First, we introduce fine and coarse grids for the computational domain. Let $\mathcal{T}^H = \{ K_i \}_{i=1}^N$ be a conforming partition of the domain $\Omega$ with mesh size $H>0$ defined by
$$H := \max_{K \in \mathcal{T}^H} \Big(\max_{x, y \in K} \abs{x-y}\Big).$$ 
We refer to this partition as the coarse grid. We denote the total number of coarse elements as $N \in \mathbb{N}^+$. Subordinate to the coarse grid, we define the fine grid partition $\mathcal{T}^h$ (with mesh size $h \ll H$) by refining each coarse element $K \in \mathcal{T}^H$ into a connected union of finer elements. We assume that the refinement above is performed such that $\mathcal{T}^h$ is also a conforming partition of the domain $\Omega$. Denote $N_c$ as the number of interior coarse grid nodes of $\mathcal{T}^H$ and we write $\{ x_i \}_{i=1}^{N_c}$ as the interior coarse nodes in the coarse grid $\mathcal{T}^H$. 

The mixed wave problem \eqref{eqn:model} can be numerically solved on the fine grid $\mathcal{T}^h$ by the lowest order Raviart-Thomas ($RT0$) finite element method. Let $(V_h,Q_h)$ be the $RT0$ finite element spaces with respect to $\mathcal{T}^h$. The approximated variational formulation reads: find $(v_h, p_h) \in V_h \times Q_h$ such that 
\begin{eqnarray}
	a(\dot{v}_h,w) - b(w,p_h) = & 0 & \quad \textcolor{cyan}{\forall w \in V_h}, \label{eqn:fine_1}\\
	(\dot{p}_h,q)_\rho + b(v_h,q)  = & (f,q) &\quad \forall q \in Q_h.\label{eqn:fine_2}
\end{eqnarray}
We remark that the solution pair $(v_h,p_h) \in V_h \times Q_h$ is served as a reference solution. In the following sections, we will construct multiscale solution $(v_{\text{ms}}, p_{\text{ms}})$ that gives a good approximation of $(v_h,p_h)$ and derive the corresponding error estimation. For an error bound of the reference solution $(v_h,p_h)$, one can apply the technique in \cite{becache2000analysis} to show that 
$$\norm{v - v_h}_{H(\text{div},\Omega)} + \norm{p-p_h}_{L^2(\Omega)} \leq Ch,$$ where $C>0$ is a constant depending on the regularity of the exact solution $(v,p)$. 

\section{Methodology}\label{sec:method}
In this section, we outline the framework of CEM-GMsFEM and introduce the construction of the multiscale spaces for approximating the fine-scale solution $(v_h, p_h)$. We emphasize that the multiscale basis functions and the corresponding spaces are defined with respect to the coarse grid $\mathcal{T}^H$. The multiscale method consists of two steps. First, we construct a multiscale space $Q_{\text{ms}}$ for approximating the pressure. Based on the space $Q_{\text{ms}}$, we construct another multiscale space $V_{\text{ms}}$ for the velocity.  We remark that these basis functions are locally supported in some coarse patches formed by some coarse elements. Once the multiscale spaces are ready, one can discretize time derivatives in the problem by finite differences and solve the resulting fully discretized problem. 

\subsection{The multiscale method}\label{sec:multiscale}
First, we introduce some notations that will be used later. Given a subset $S \subset \Omega$, we define $V_{h,0}(S) := \{ v \in V_h \cap H(\text{div}; S): v \cdot \mathbf{n}_S = 0 ~\text{on } \partial S\}$ and $Q_h(S) := Q_h \cap L^2(S)$, where $\mathbf{n}_S$ is the unit outward normal vector with respect to the boundary $\partial S$. 

The multiscale solution $(v_{\text{ms}}, p_{\text{ms}}) \in V_{\text{ms}} \times Q_{\text{ms}}$ is obtained by solving the variational formulation
\begin{eqnarray}
	a(\dot{v}_{\text{ms}},w) - b(w,p_{\text{ms}}) = & 0 & \quad \forall w \in V_{\text{ms}}, \label{eqn:ms_var_1}\\
	(\dot{p}_{\text{ms}},q)_\rho + b(v_{\text{ms}},q)  = & (f,q) &\quad \forall q \in Q_{\text{ms}}. \label{eqn:ms_var_2}
\end{eqnarray}
We will detail the constructions for the multiscale spaces in the next sections. 

\subsubsection{Construction of pressure basis functions}
We present the construction of the multiscale space $Q_{\text{ms}}$ for pressure. For each coarse element $K_i \in \mathcal{T}^H$, consider the following local spectral problem over $K_i$: find $(\phi_j^{i}, p_j^{i}) \in V_{h,0}(K_i) \times Q_h(K_i)$ and $\lambda_j^{i} \in \mathbb{R}$ such that 
\begin{eqnarray}
	a(\phi_j^{i}, v) - b(v,p_j^{i}) = & 0 & \quad \forall v \in V_{h,0}(K_i), \label{eqn:sp_1}\\
	b(\phi_j^{i}, q) = & \lambda_j^{i} s_i(p_j^{i},q) & \quad \forall q \in Q_h(K_i), \label{eqn:sp_2}
\end{eqnarray}
for $j = 1,\cdots, L_i$, where $L_i \in \mathbb{N}^+$ is a local parameter depending on the grids $\mathcal{T}^H$ and $\mathcal{T}^h$. Here, the bilinear form $s_i : Q_h(K_i) \times Q_h(K_i) \to \mathbb{R}$ is defined as follows: 
$$ s_i(p,q) := \int_{K_i} \tilde \kappa pq ~dx, \quad \text{where } \tilde \kappa := \kappa \sum_{j=1}^{N_c} \abs{\nabla \chi_j}^2,$$
and $\{ \chi_j\}_{j=1}^{N_c}$ is a set of standard multiscale partition of unity. In particular, given an interior coarse grid node $x_j$, the function $\chi_j$ is defined as the solution to the following system over the coarse neighborhood $\omega_j := \bigcup \{ K \in \mathcal{T}^H: x_j \in \partial K\}$
\begin{eqnarray*}
-\divg{(\kappa \nabla \chi_j)} = &0& \quad \text{in all } K \subset \omega_j, \\
\chi_j = &g_j& \quad \text{on } \partial K \setminus \partial \omega_j ~ (\text{for all } K \subset \omega_j), \\
\chi_j = &0& \quad \text{on } \partial \omega_j,
\end{eqnarray*}
where $g_j$ is a linear continuous function on all edges of $\partial K $. Assume that $s_i(p_j^{i}, p_j^{i}) = 1$ and we arrange the eigenvalues in ascending order such that $0 \leq \lambda_1^{i} \leq \cdots \leq \lambda_{L_i}^i$. For each $i \in \{1, 2, \cdots, N\}$, choose the first $J_i \in \mathbb{N}^+$ ($1 \leq J_i \leq L_i$) eigenfunctions $\{p_j^{i}\}_{j=1}^{J_i}$ corresponding the first $J_i$ smallest eigenvalues. Then, we define the multiscale space $Q_{\text{ms}}$ for pressure as follows: 
$$ Q_{\text{ms}} := \text{span} \{ p_j^{i}: i = 1,\cdots, N, ~ j = 1,\cdots, J_i \}.$$

\subsubsection{Construction of velocity basis functions}
In this section, we present the construction of the multiscale space $V_{\text{ms}}$ for velocity. 
To define the velocity basis, we introduce the operator $\pi: Q_h \to Q_{\text{ms}}$ as follows: 
$$\pi q := \pi (q) = \sum_{i=1}^N \sum_{j=1}^{J_i} s_i(p_j^{i},q) p_j^{i} \quad \tforall q \in Q_h.$$
Next, we define the bilinear form $s: Q_h \times Q_h \to \mathbb{R}$ as $s(p,q) := \sum_{i=1}^N s_i(p,q)$ for all $p,q \in Q_h$. Note that the operator $\pi : Q_h \to Q_{\text{ms}}$ is the projection of $Q_h$ onto the multiscale space $Q_{\text{ms}}$ with respect to the inner product $s(\cdot,\cdot)$. We denote the norm induced by this inner product $s(\cdot,\cdot)$ as $\norm{\cdot}_s$. 

For a given coarse element $K_i \in \mathcal{T}^H$ and a parameter $\ell \in \mathbb{N}^+$, we define $K_{i,\ell}$ to be the oversampled region obtained by enlarging $\ell$ layers from $K_i$. Specifically, we have 
$$K_{i,0} := K_i, \quad K_{i,\ell} := \bigcup \big \{ K\in \mathcal{T}^H: K \cap \overline{K_{i,\ell-1}} \neq \emptyset \big \}, \quad \ell = 1,2, \cdots. $$
For simplicity, we denote $K_i^+$ the oversampled region. For each eigenfunction $p_j^{i} \in Q_{\text{ms}}$ obtained from \eqref{eqn:sp_1}-\eqref{eqn:sp_2}, we define the multiscale basis for velocity $\psi_{j,\text{ms}}^{i} \in V_{h,0}(K_i^+)$ to be the solution of the following system: 
\begin{eqnarray}
	a(\psi_{j,\text{ms}}^{i}, v) - b(v, q_{j,\text{ms}}^{i}) = & 0 & \quad \forall v \in V_{h,0}(K_i^+), \label{eqn:msv_1}\\
	s(\pi q_{j,\text{ms}}^{i},\pi q) + b(\psi_{j,\text{ms}}^{i},q) = &s(p_{j}^{i},q) & \quad \forall q \in Q_h(K_i^+). \label{eqn:msv_2}
\end{eqnarray}
Then, the multiscale space for velocity is defined as 
$$ V_{\text{ms}} := \text{span} \{ \psi_{j,\text{ms}}^{i}: i = 1,\cdots, N, ~ j = 1,\cdots, J_i\}.$$

\subsection{Discretizations}
In this section, we discuss the fully discretization of the problem. The multiscale spaces $V_{\text{ms}}$ and $Q_{\text{ms}}$ obtained in Section \ref{sec:multiscale} are constructed in the spirit of CEM-GMsFEM. To simplify the notation, we assume that 
$$ V_{\text{ms}} = \text{span} \{ \psi_i\}_{i=1}^M \quad \text{and} \quad Q_{\text{ms}} = \text{span} \{ p_i \}_{i=1}^M,$$
where $M := \sum_{i=1}^{N} J_i$. Then, we obtain the following matrix representation from the variational formulation
\begin{eqnarray}
\mathcal{M}_v \dot{\mathbf{v}} - \mathcal{R} \mathbf{p} & = & \mathbf{0}, \\
\mathcal{M}_p \dot{\mathbf{p}} + \mathcal{R}^T \mathbf{v} & = & \mathbf{f},
\end{eqnarray}
where $\mathbf{0} \in \mathbb{R}^M$ is the zero vector in $\mathbb{R}^M$. Moreover, we have the following definitions of the matrices: 
$$ \mathcal{M}_v := \big( a(\psi_i,\psi_j) \big) \in \mathbb{R}^{M \times M}, \quad \mathcal{M}_p := \big( (p_i,p_j) \big) \in \mathbb{R}^{M \times M}, $$
$$ \mathcal{R} := \big( b(\psi_i,p_j) \big) \in \mathbb{R}^{M \times M}, \quad \text{and} \quad \mathbf{f} := \big( (f,p_i) \big) \in \mathbb{R}^M.$$
Note that $\mathbf{v} := \mathbf{v}(t) = \big(\mathbf{v}_i(t) \big)_{i=1}^M \in \mathbb{R}^M$ and $\mathbf{p} := \mathbf{p}(t) = \big(\mathbf{p}_i(t)\big)_{i=1}^M \in \mathbb{R}^M$ are the vectors of coefficients for the approximations $v_{\text{ms}}$ and $p_{\text{ms}}$. More precisely, we have 
$$ v_{\text{ms}} = \sum_{i=1}^M \mathbf{v}_i(t) \psi_i \quad \text{and} \quad p_{\text{ms}} = \sum_{i=1}^M \mathbf{p}_i(t) p_i.$$

For the time discretization, we simply replace the continuous time derivatives by the forward difference in time with a given time step $\tau>0$. In particular, the velocity term will be approximated at $t_n = n \tau$ and the pressure term will be approximated at $t_{n+\frac{1}{2}} = \left(n+\frac{1}{2}\right)\tau$ for $n \in \{0,1,\cdots, N_T \}$ with $T = N_T \tau$. We remark that the time step $\tau$ will be chosen such that $N_T \in \mathbb{N}^+$. 
It leads to the following fully discretized system: given $(\mathbf{v}^n, \mathbf{p}^{n+\frac{1}{2}})$ and for $n \geq 0$, find $(\mathbf{v}^{n+1}, \mathbf{p}^{n+\frac{3}{2}})$ such that 
\begin{eqnarray}
\mathcal{M}_v \frac{\mathbf{v}^{n+1} - \mathbf{v}^n}{\tau} - \mathcal{R} \mathbf{p}^{n+\frac{1}{2}} & = & \mathbf{0}, \label{eqn:time_1} \\
\mathcal{M}_p \frac{\mathbf{p}^{n+\frac{3}{2}}- \mathbf{p}^{n+\frac{1}{2}}}{\tau} + \mathcal{R}^T \mathbf{v}^{n+1} & = & \mathbf{f}^{n+1}, \label{eqn:time_2}
\end{eqnarray}
where $\mathbf{v}^n := \mathbf{v}(t_n)$, $\mathbf{p}^{n+\frac{1}{2}} := \mathbf{p}\left(t_{n+\frac{1}{2}}\right)$, and $\mathbf{f}^n := \big( f(t_n), p_i \big)_{i=1}^M$. 
We remark that $\mathbf{v}^0$ (resp. $\mathbf{p}^{\frac{1}{2}}$) is the vector of coefficient of the projection of the initial condition $h_v$ (resp. $h_p$) with respect to the bilinear form $a(\cdot, \cdot)$ (resp. $(\cdot, \cdot)_\rho$). 
We remark that the stability estimate for $\tau$ can be obtained by standard techniques and the inverse estimate, see for example \cite{gibson2014multiscale}. 

\section{Stability and convergence analysis}\label{sec:analysis}
In this section, we present the results of stability and convergence analysis for the mixed CEM-GMsFEM established in Section \ref{sec:method}. To this aim, we first introduce some notations that will be used in this section:  
$$(u,v)_V := \int_\Omega u\cdot v ~ dx, \quad  \norm{v}_V := \sqrt{(v,v)_V}, \quad \norm{v}_a := \sqrt{a(v,v)}, \quad \text{and} \quad \norm{p}_{\rho} := \sqrt{(p,p)_\rho}, $$ 
for all $u, v \in V_h$, and $p \in Q_h$. 
We remark that the norms $\norm{\cdot}_{\rho}$ and $\norm{\cdot}_s$ are equivalent to the standard $L^2$-norm $\norm{\cdot}_{L^2(\Omega)}$. 
For $p \in Q_h$, we have 
$$\rho_{\min} \norm{p}_{L^2(\Omega)}^2\leq \norm{p}_{\rho}^2 \leq \norm{\rho}_{L^\infty(\Omega)} \norm{p}_{L^2(\Omega)}^2.$$
One may show the equivalence between $\norm{\cdot}_s$ and $\norm{\cdot}_{L^2(\Omega)}$ provided $\kappa \in [\kappa_{\min}, \kappa_{\max}]$. Further, we denote $a \lesssim b$ if there is a generic constant $C>0$ such that $a \leq C b$. We write $a \lesssim_T b$ if there exists a generic constant $C = C(T)>0$, depending on $T$, such that $a \leq C(T)b$. 
 
\subsection{Energy conservation and stability}
In this section, we prove the property of energy conservation and the stability of the proposed scheme \eqref{eqn:ms_var_1}-\eqref{eqn:ms_var_2}. In particular, we show the following proposition.
\begin{proposition}\label{prop:sta}
Let $(v_{\text{ms}}, p_{\text{ms}}) \in V_{\text{ms}} \times Q_{\text{ms}}$ be the solution of \eqref{eqn:ms_var_1}-\eqref{eqn:ms_var_2}. Then, the following property of energy conservation holds
$$\frac{d}{dt} \big( \norm{v_{\text{ms}}}_a^2 + \norm{p_{\text{ms}}}_{\rho}^2 \big) = 0 \quad \text{if } f \equiv 0.$$
Moreover, the following estimate holds: 
$$ \max_{0 \leq t \leq T} \Big ( \norm{v_{\text{ms}}(t,\cdot)}_a^2 + \norm{p_{\text{ms}}(t,\cdot)}_{\rho}^2 \Big) \lesssim \bigg( \norm{h_v}_a^2 + \norm{h_p}_{\rho}^2 + \int_0^T \norm{\rho^{-1} f }_{\rho}^2 dt\bigg).$$
\end{proposition}

\begin{proof}
Take $v = v_{\text{ms}}$ in \eqref{eqn:ms_var_1} and $q = p_{\text{ms}}$ in \eqref{eqn:ms_var_2}. Adding two equations, we obtain 
\begin{eqnarray}
	a(\dot{v_{\text{ms}}}, v_{\text{ms}}) + (\dot{p_{\text{ms}}}, p_{\text{ms}})_\rho = \frac{1}{2}\frac{d}{dt} \big( \norm{v_{\text{ms}}}_a^2 + \norm{p_{\text{ms}}}_{\rho}^2 \big) = (\rho^{-1} f,p_{\text{ms}})_{\rho} = 0.
\end{eqnarray}
Using the Cauchy-Schwarz inequality, we have 
$$\frac{d}{dt} \big( \norm{v_{\text{ms}}}_a^2 + \norm{p_{\text{ms}}}_{\rho}^2 \big) \leq 2 \norm{\rho^{-1} f}_{\rho} \norm{p_{\text{ms}}}_{\rho}.$$
It implies that 
\begin{eqnarray}
\norm{v_{\text{ms}}(t,\cdot)}_a^2 + \norm{p_{\text{ms}}(t,\cdot)}_{\rho}^2 \leq 2 \max_{0 \leq t \leq T} \norm{p_{\text{ms}}(t,\cdot)}_{\rho} \bigg( \int_0^T \norm{\rho^{-1} f}_{\rho} dt \bigg)+ \norm{h_v}_a^2 + \norm{h_p}_{\rho}^2.
\label{eqn:derive_1}
\end{eqnarray}
Using Young's and Jensen's inequalities, we obtain 
$$\norm{v_{\text{ms}}(t,\cdot)}_a^2 + \norm{p_{\text{ms}}(t,\cdot)}_{\rho}^2 \leq \frac{1}{2} \max_{0 \leq t \leq T} \norm{p_{\text{ms}}(t,\cdot)}_{\rho}^2 + 2\bigg(\int_0^T \norm{\rho^{-1} f}_{\rho} dt\bigg)^2 + \norm{h_v}_a^2 + \norm{h_p}_{\rho}^2,$$
and therefore 
\begin{eqnarray}
\max_{0 \leq t \leq T} \Big( \norm{v_{\text{ms}}(t,\cdot)}_a^2 + \norm{p_{\text{ms}}(t,\cdot)}_{\rho}^2 \Big) \leq 4 \bigg(  \norm{h_v}_a^2 + \norm{h_p}_{\rho}^2 + \int_0^T \norm{\rho^{-1} f }_{\rho}^2 dt \bigg).
\label{eqn:derive_2} 
\end{eqnarray}
This completes the proof.
\end{proof}

\begin{remark}
The technique showing \eqref{eqn:derive_2} from \eqref{eqn:derive_1} in the proof above will be employed in the convergence analysis below. 
\end{remark}

\subsection{Convergence analysis}
In this section, we show the convergence result of the semi-discretized scheme. We define $(\pp{v}, \pp{p}) \in V_{\text{ms}} \times Q_{\text{ms}}$ to be the {\it multiscale projection} of a given pair $(v,p) \in V_h \times Q_h$ (with $\int_\Omega p~ dx =0$ and satisfying $a(v,w) - b(w,p) = 0$ for any $w \in V_h$) if 
\begin{eqnarray}
\begin{split}
	a(\pp{v}, w) - b(w,\pp{p}) &= 0  & \quad \forall w \in V_{\text{ms}}, \\
	b(\pp{v},q) &= b(v,q) & \quad \forall q \in Q_{\text{ms}}. 
\end{split}
\label{eqn:p1}
\end{eqnarray}
We have the following auxiliary result for the multiscale projection. 

\begin{lemma} \label{lem:ms-cem}
For any $v \in V_h$, the following estimate holds: 
\begin{eqnarray}
\norm{v - \pp{v}}_a \lesssim H \Lambda^{-1/2} . 
\label{eqn:p4}
\end{eqnarray}
where $\Lambda := \displaystyle{\min_{1 \leq i \leq N} \lambda_{J_i +1}^i}$ and $\{ \lambda_j^i \}$ are the eigenvalues obtained from \eqref{eqn:sp_1}-\eqref{eqn:sp_2}. 
\end{lemma}
\begin{proof}
For any $v \in V_h$, we define $\beta \in Q_h$ (with $\int_\Omega \beta ~ dx = 0$) such that 
\begin{eqnarray*}
 b(w, \beta) = a(v - \pp{v},w)
\quad 
\forall w \in V_h.
\end{eqnarray*}
Denote $z = v- \pp{v}$. Then, $(z,\beta) \in V_h \times Q_h$ satisfies the following system: 
\begin{eqnarray*}
\begin{split}
	a(z, w) - b(w,\beta) = & 0 & \quad \forall w \in V_{h}, \\
	b(z, q) = & \left( \nabla \cdot (v - \pp{v} ) , q \right) & \quad \forall q \in Q_{h}.
\end{split}
\end{eqnarray*}
Hence, the following estimate holds: 
$$ \norm{z - \pp{z}}_a \lesssim H \norm{\nabla \cdot (v - \pp{v})}_{L^2(\Omega)},$$
using the result of \cite[Theorem 1]{chung2018cemmixed}. Therefore, we have 
\begin{eqnarray*}
\norm{v - \pp{v}}_a^2 &=& a(z, v- \pp{v}) = a(z- \pp{z}, v- \pp{v}) \\
& \leq & \norm{z- \pp{z}}_a \norm{v - \pp{v}}_a \\
& \lesssim & H \norm{\nabla \cdot (v - \pp{v})}_{L^2(\Omega)} \norm{v - \pp{v}}_a.
\end{eqnarray*}
On the other hand, since $b(v - \pp{v}, q) = s(\tilde \kappa^{-1} \nabla \cdot ( v- \pp{v}), q) = 0$ for all $q \in Q_{\text{ms}}$, there exists a set of real numbers $\{ c_j^i \}$ such that 
$$ \tilde \kappa^{-1} \nabla \cdot (v- \pp{v}) = \sum_{i=1}^N \sum_{j>J_i} c_j^i p_j^i.$$
Then, by the orthogonality of the eigenfunctions $\{p_j^i\}$ and \eqref{eqn:sp_1}-\eqref{eqn:sp_2}, we have 
\begin{eqnarray*}
\norm{\nabla \cdot (v - \pp{v})}_{L^2(\Omega)}^2 & = & \norm{\tilde \kappa^{-1/2} \nabla \cdot ( v- \pp{v})}_s^2 \\
& \lesssim & \sum_{i=1}^N \sum_{j>J_i} (c_j^i)^2 \norm{p_j^i}_s^2 ~ \leq ~ \Lambda^{-1} \sum_{i=1}^N \sum_{j>J_i} (c_j^i)^2 \norm{\phi_j^i}_a^2.
\end{eqnarray*}
This completes the proof. 
\end{proof}

The main result in this research reads as follows.
\begin{theorem}\label{thm:conv}
Suppose that $(v_h, p_h)$ is the solution to the system \eqref{eqn:fine_1}-\eqref{eqn:fine_2} and $(v_{\text{ms}}, p_{\text{ms}})$ is the solution to \eqref{eqn:ms_var_1}-\eqref{eqn:ms_var_2}. Then, the following estimate holds
\begin{eqnarray}
	\max_{0 \leq t \leq T} \Big( \norm{\pp{v_h} - v_{\text{ms}}}_a^2 + \norm{\pp{p_h} - p_{\text{ms}}}_\rho^2 \Big) \lesssim_T H^2\left ( \Lambda^{-1}+ \norm{\dot{\mathcal{F}}}_{L^2(0,T; L^2(\Omega))}^2 + \norm{\ddot{\mathcal{F}}}_{L^2(0,T;L^2(\Omega))}^2 \right ) \label{eqn:main_result}
\end{eqnarray}
where $(\pp{v_h}, \pp{p_h}) \in V_{\text{ms}} \times Q_{\text{ms}}$ is the multiscale projection of $(v_h,p_h)$ and $\mathcal{F} := f -  \rho \dot{p_h}$. 
\end{theorem}

\begin{proof}
Subtracting \eqref{eqn:ms_var_1} from \eqref{eqn:fine_1} and \eqref{eqn:ms_var_2} from \eqref{eqn:fine_2}, one obtains
\begin{eqnarray*}
	a(\dot{v_h} - \dot{v_{\text{ms}}}, w ) - b(w, p_h - p_{\text{ms}}) = & 0 & \quad \forall w \in V_{\text{ms}}, \\
	(\dot{p_h} - \dot{p_{\text{ms}}}, q )_\rho + b(v_h - v_{\text{ms}}, q) = & 0 & \quad \forall q \in Q_{\text{ms}}.
\end{eqnarray*}
Rewriting the system above, it implies that 
\begin{eqnarray*}
	a(\dot{\pp{v_h}} - \dot{v_{\text{ms}}}, w ) - b(w, \pp{p_h} - p_{\text{ms}}) = & a(\dot{\pp{v_h}} - \dot{v_h}, w ) - b(w, \pp{p_h} - p_h) & \quad \forall w \in V_{\text{ms}}, \\
	(\dot{\pp{p_h}} - \dot{p_{\text{ms}}}, q )_\rho + b(\pp{v_h} - v_{\text{ms}}, q) = & (\dot{\pp{p_h}} - \dot{p_h}, q )_\rho + b(\pp{v_h} - v_h, q) & \quad \forall q \in Q_{\text{ms}}. 
\end{eqnarray*}
Take $w = \pp{v_h} - v_{\text{ms}} \in V_{\text{ms}}$ and $q = \pp{p_h} - p_{\text{ms}} \in Q_{\text{ms}}$ in the system above. Then, adding two equations together, we obtain the following equality
$$ \text{LHS} := a(\dot{\pp{v_h}} - \dot{v_{\text{ms}}}, \pp{v_h} - v_{\text{ms}}) + (\dot{\pp{p_h}} - \dot{p_\text{ms}}, \pp{p_h} - p_{\text{ms}})_\rho = \text{RHS},$$
where
\begin{eqnarray*}
	\text{RHS} &:=&  a(\dot{\pp{v_h}} - \dot{v_h}, \pp{v_h} - v_{\text{ms}} ) - b(\pp{v_h} - v_{\text{ms}}, \pp{p_h} - p_h) + (\dot{\pp{p_h}} - \dot{p_h}, \pp{p_h} - p_{\text{ms}} )_\rho\\
	& ~ &  + b(\pp{v_h} - v_h, \pp{p_h} - p_{\text{ms}}) \\
	& = & a(\dot{\pp{v_h}} - \dot{v_h} - (\pp{v_h} - v_h), \pp{v_h} - v_{\text{ms}} ) 
	+ (\dot{\pp{p_h}} - \dot{p_h}, \pp{p_h} - p_{\text{ms}} )_\rho.
\end{eqnarray*}
Here, the properties of multiscale projection \eqref{eqn:p1} are used to simplify the expression above. Hence, we obtain by Cauchy-Schwarz inequality
\begin{eqnarray*}
	\text{LHS} & = & \frac{1}{2} \frac{d}{dt}\Big( \norm{\pp{v_h} - v_{\text{ms}}}_a^2 + \norm{\pp{p_h} - p_{\text{ms}}}_\rho^2 \Big)  = \text{RHS} \\
	& \leq & (\norm{\dot{\pp{v_h}} - \dot{v_h}}_a + \norm{\pp{v_h} - v_h}_a) \norm{\pp{v_h} - v_{\text{ms}}}_a 
	+ \norm{\dot{\pp{p_h}} - \dot{p_h}}_\rho  \norm{\pp{p_h} - p_{\text{ms}}}_\rho.
\end{eqnarray*}
It implies that
\begin{eqnarray*}
	\norm{\pp{v_h} - v_{\text{ms}}}_a^2 + \norm{\pp{p_h} - p_{\text{ms}}}_\rho^2 & \leq & 2 \left( \max_{0 \leq t \leq T} \norm{\pp{v_h} - v_{\text{ms}}}_a \right) \int_0^T \norm{\dot{\pp{v_h}} - \dot{v_h}}_a + \norm{\pp{v_h} - v_h}_a~ dt\\
	& ~ & + 2 \left ( \max_{0 \leq t \leq T} \norm{\pp{p_h} - p_{\text{ms}}}_\rho \right) \int_0^T \norm{\dot{\pp{p_h}} - \dot{p_h}}_\rho ~ dt.
\end{eqnarray*}
Using the same technique as that of proving Proposition \ref{prop:sta}, we obtain 
\begin{eqnarray}
\max_{0 \leq t \leq T} \Big( \norm{\pp{v_h} - v_{\text{ms}}}_a^2 + \norm{\pp{p_h} - p_{\text{ms}}}_\rho^2 \Big) \lesssim \int_0^T  \norm{\pp{v_h} - v_h}_a^2 + \norm{\dot{\pp{v_h}} - \dot{v_h}}_a^2 
+ \norm{\dot{\pp{p_h}} - \dot{p_h}}_\rho^2 dt.
\label{eqn:derive_5}
\end{eqnarray}
Next, we analyze the terms $\norm{\pp{v_h} - v_h}_a$, $\norm{\dot{\pp{v_h}} - \dot{v_h}}_a^2$, 
and $\norm{\dot{\pp{p_h}} - \dot{p_h}}_\rho^2$. 
Using the result of Lemma \ref{lem:ms-cem}, we obtain
\begin{eqnarray} 
\norm{\pp{v_h} - v_h}_a^2 \lesssim H^2 \Lambda^{-1}.
\label{eqn:derive_5-1}
\end{eqnarray}
Note that the fine-scale solution $(v_h, p_h)$ satisfies the following system
\begin{eqnarray*}
	a(\dot{v_h} , w) - b(w,p_h) = & 0 & \quad \forall w \in V_h, \\
	b(\dot{v_h}, q) = & (\dot{\mathcal{F}},q) & \quad \forall q \in Q_h,
\end{eqnarray*}
where $\mathcal{F}= f - \rho \dot{p_h}$. 
Using the properties of the multiscale projection \eqref{eqn:p1}, we obtain 
\begin{eqnarray*}
	a(\dot{\pp{v_h}} , w) - b(w,\pp{p_h}) = & 0 & \quad \forall w \in V_{\text{ms}}, \\
	b(\dot{\pp{v_h}}, q) = & (\dot{\mathcal{F}},q) & \quad \forall q \in Q_{\text{ms}}.
\end{eqnarray*}
Then, by the error estimate in \cite[Theorem 1]{chung2018cemmixed}, one may obtain the following estimate
\begin{eqnarray}
	\norm{\dot{\pp{v_h}} - \dot{v_h}}_a^2 + \norm{\dot{\pp{p_h}} - \dot{p_h}}_\rho^2 
	\lesssim H^2 \left ( \norm{\dot{\mathcal{F}}}_{L^2(\Omega)}^2 +  \norm{\ddot{\mathcal{F}}}_{L^2(\Omega)}^2 \right). \label{eqn:derive_6}
\end{eqnarray}

Combining \eqref{eqn:derive_5}, \eqref{eqn:derive_5-1}, and \eqref{eqn:derive_6} yields the desired estimate. This completes the proof. 
\end{proof}
\begin{remark}
From the result \eqref{eqn:p4} and the inequality \eqref{eqn:main_result}, one may easily conclude that 
$\norm{v_h - v_{\text{ms}}}_a = O(H)$. 
Moreover, using the technique in \cite[Theorem 5.4]{chung2016mixed} one may show that $p_h$ satisfies
$$ \norm{p_h - \pp{p_h}}_{\rho} \lesssim \Lambda^{-1} \norm{\dot{v_h}}_a \implies \norm{p_h-p_{\text{ms}}}_\rho \lesssim \Lambda^{-1}\norm{\dot{v_h}}_a + O(H).$$
\end{remark}

\section{Numerical experiments}\label{sec:numerics}
In this section, we perform some numerical experiments using the proposed multiscale method to solve the wave equation in mixed formulation. Let the computational domain and the density be $\Omega = (0,1)^2$, $T = 0.4$, and $\rho \equiv 1$. 
In the simulation, we use (uniform) rectangular mesh to perform the spatial discretization and set the fine mesh size to be $h = \sqrt{2}/400$. We will specify the coarse mesh size in the examples below. For the time discretization, we set the time step to be $\tau = 10^{-4}$. 
In both examples, let the initial conditions be $h_v = 0$ and $h_p = 0$. Recall that $\ell \in \mathbb{N}$ is the number of oversampling layers used to perform the constraint energy minimization. We set $J_i = J$ (uniformly) to form the auxiliary multiscale space $Q_{\text{ms}}$. 

We will use the following quantities to measure the performance of the proposed method: 
$$ e_{\text{vel}} := \frac{\norm{v_h - v_{\text{ms}}}_a}{\norm{v_h}_a} \quad \text{and} \quad e_{\text{pre}} := \frac{\norm{p_h - p_{\text{ms}}}_{\rho}}{\norm{p_h}_{\rho}},$$
where $(v_h, p_h)$ is the fine-scale solution computed on $\mathcal{T}^h$ and $(v_{\text{ms}}, p_{\text{ms}})$ is the multiscale solution obtained by solving \eqref{eqn:ms_var_1}-\eqref{eqn:ms_var_2}. 

\begin{example}[Heterogeneous model] \label{exp:mw1}
In this example, we consider the case with high-contrast permeability field over the computational domain. See Figure \ref{fig:exp_1_kappa} for an illustration of this permeability field. The source function $f:\mathbb{R}^2 \to \mathbb{R}$ is set to be 
$$ f(x) = \left \{ \begin{array}{rc}
1 & x \in [0, 1/8]^2, \\
-1 & x \in [7/8, 1]^2, \\
0 & \text{otherwise}. \end{array} \right . $$
Clearly, it holds that $f \in L^2(\Omega)$ and it satisfies $\int_\Omega f ~dx = 0$.
\begin{figure}[h!]
\centering
\mbox{
\subfigure[$\kappa$ in Example \ref{exp:mw1}.]{
\includegraphics[width = 2.7in]{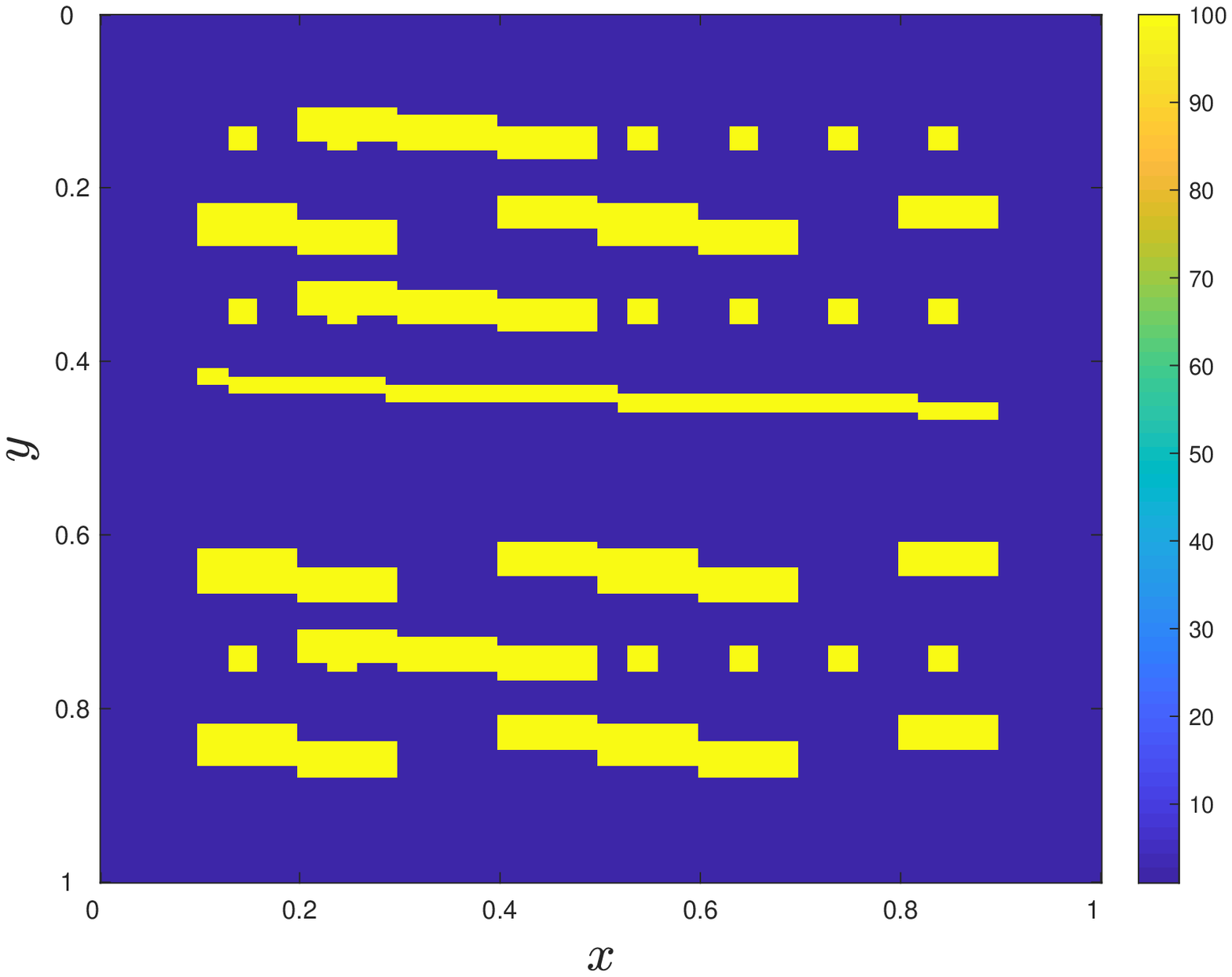}
\label{fig:exp_1_kappa}}
\subfigure[$\kappa$ in Example \ref{exp:mw2}.]{
\includegraphics[width = 2.7in]{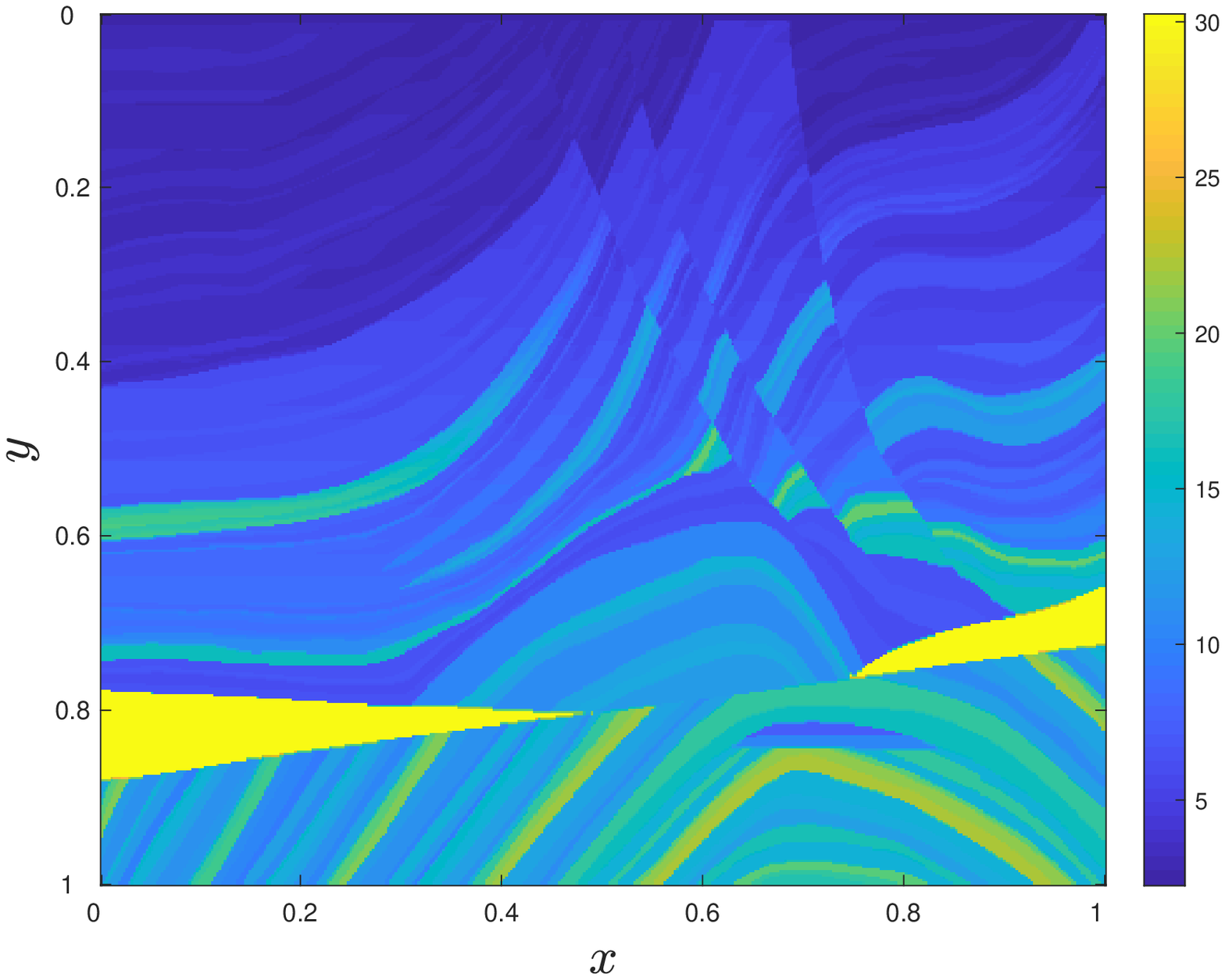}
\label{fig:exp_2_kappa}}
}
\caption{Permeability fields used in the simulation.}
\end{figure}
The solution profiles of the reference solution $(v_h, p_h)$ and the multiscale solution $(v_{\text{ms}}, p_{\text{ms}})$ with $J = 5$ and $\ell = 4$ at the terminal time $t=T$ are reported in Figures \ref{fig:ref_exp1} and \ref{fig:ms_exp1}, respectively. 
    \begin{figure}[h!]
        \centering
        \mbox{
        \subfigure[Pressure $p_h$.]{
        \includegraphics[width=2.7in]{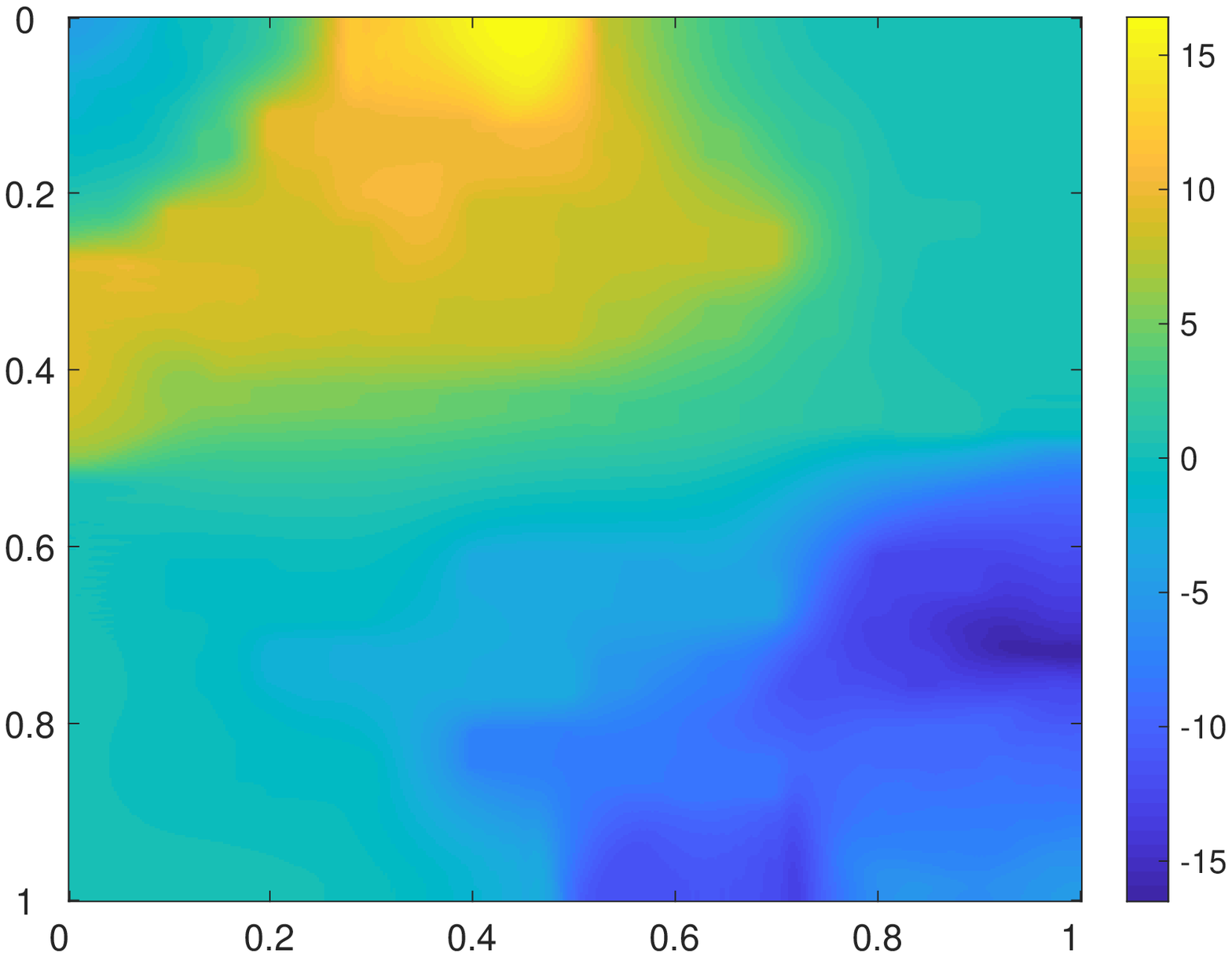} 
        \label{fig:ref_p_exp1}}
        \subfigure[Velocity $v_h$.]{
        \includegraphics[width=2.7in]{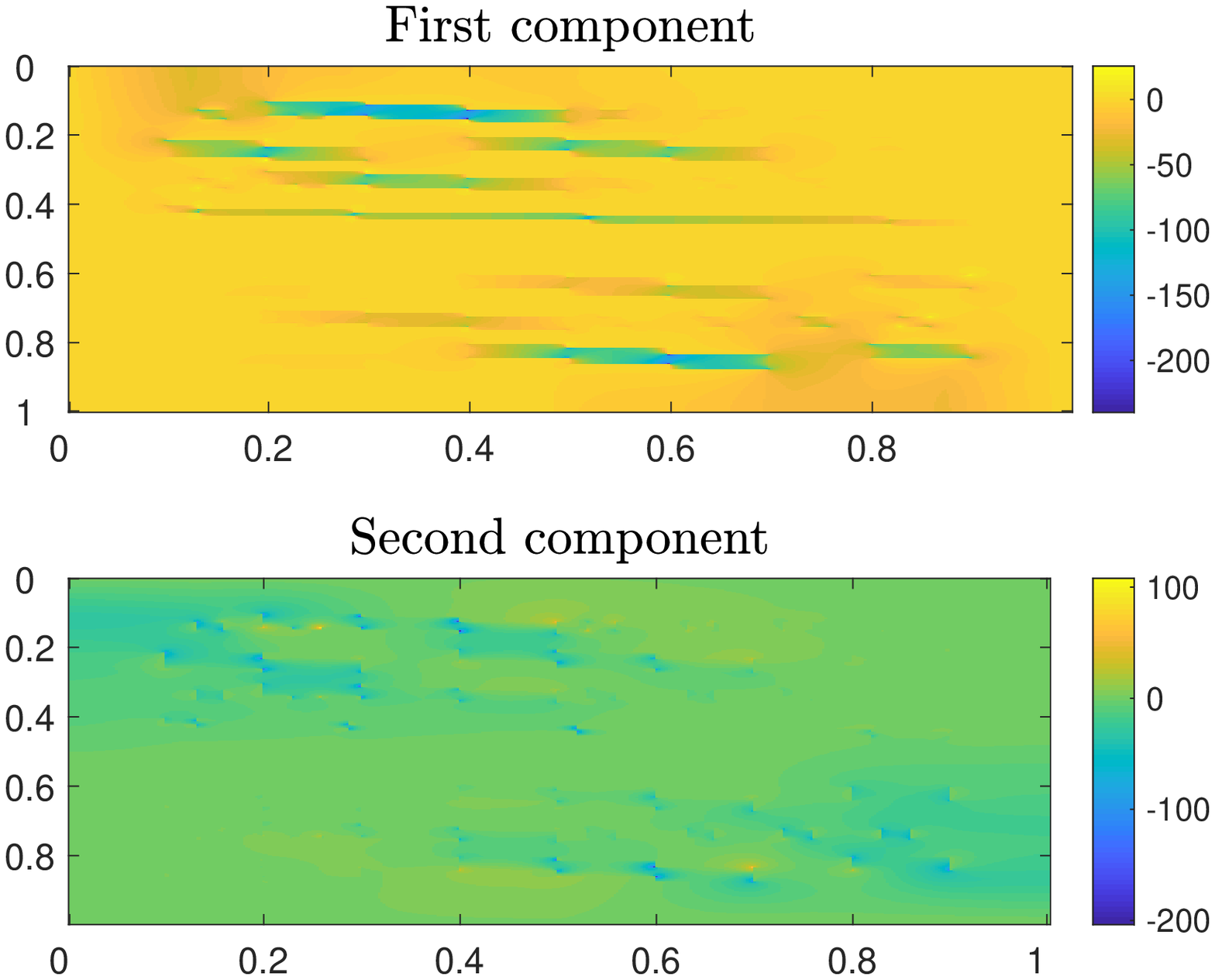}
        \label{fig:ref_v_exp1}}
     	}   
	\caption{Reference solution at $t = T$ in Example \ref{exp:mw1}.}
	\label{fig:ref_exp1}
    \end{figure}
    
    \begin{figure}[h!]
        \centering
        \mbox{
        \subfigure[Pressure $p_{\text{ms}}$.]{
        \includegraphics[width=2.7in]{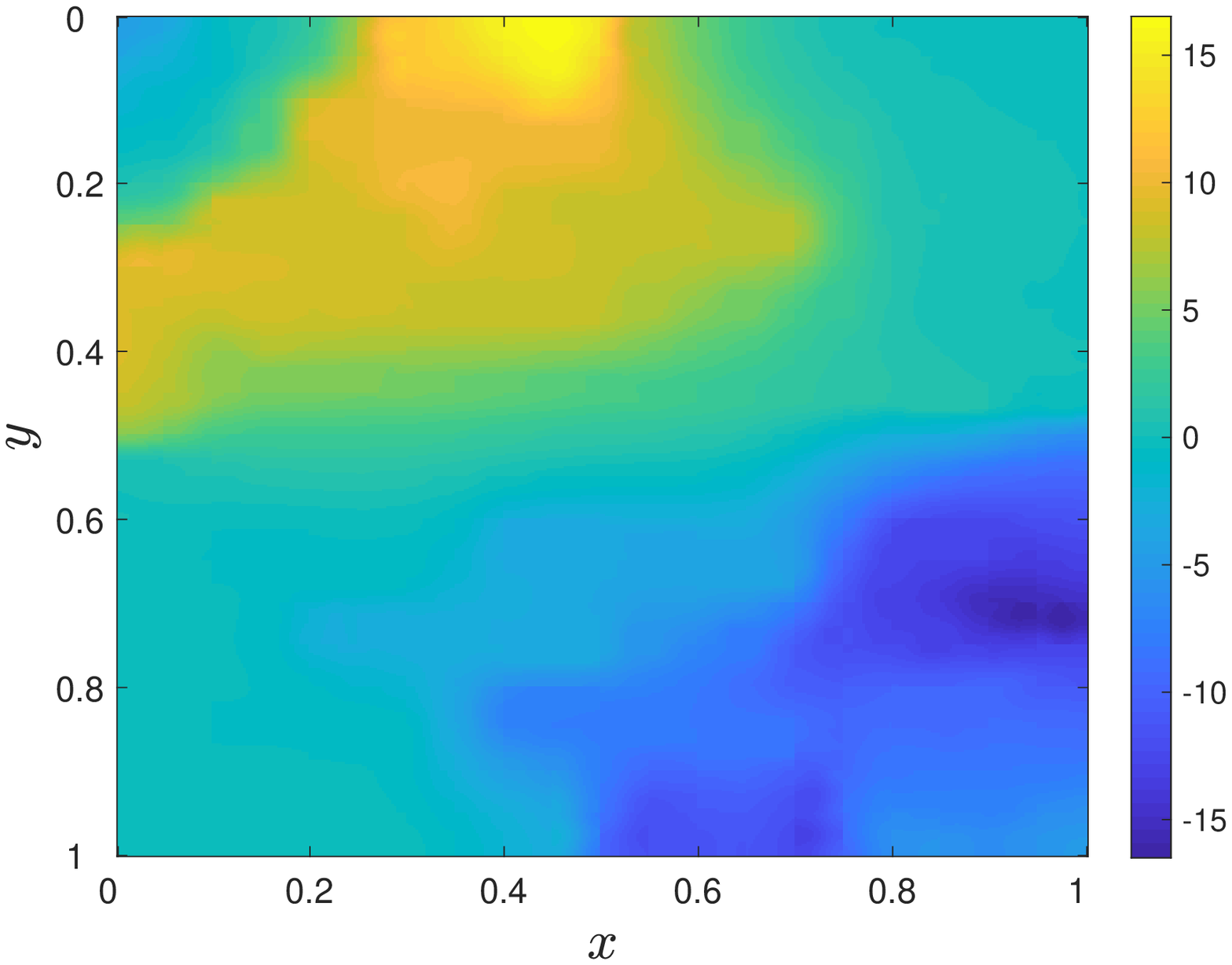}
        \label{fig:ms_p_exp1}}
        \subfigure[Velocity $v_{\text{ms}}$.]{
        \includegraphics[width=2.7in]{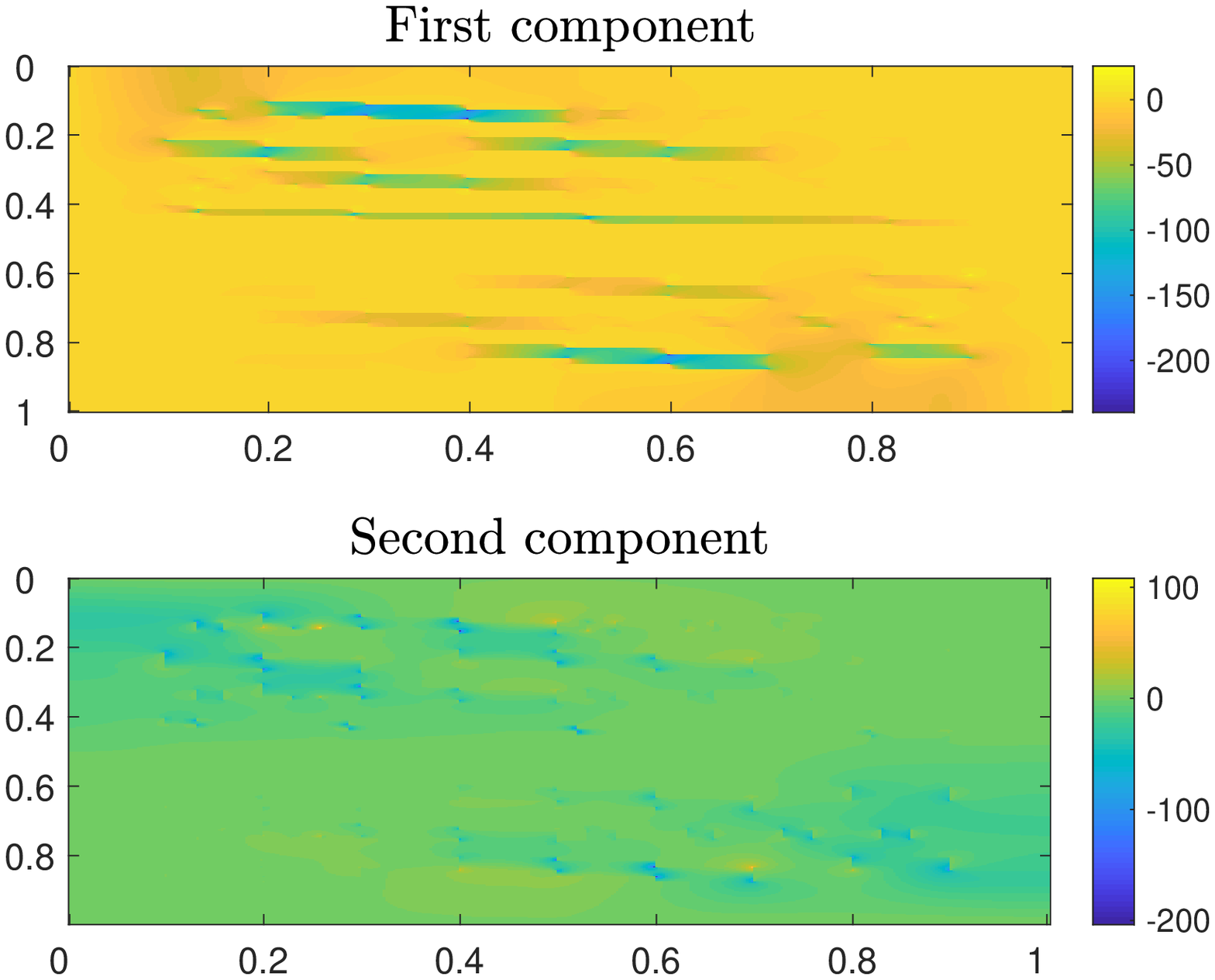}
        \label{fig:ms_v_exp1}}
        }
        \caption{Multiscale approximation at $t = T$ in Example \ref{exp:mw1}. $H = \sqrt{2}/20$; $J = 5$; $\ell = 3$.}
        \label{fig:ms_exp1}
    \end{figure}
We explore the efficiency of the proposed method by showing the errors in velocity and pressure. Tables \ref{tab:pe_exp1} and \ref{tab:ve_exp1} show the quantities $e_{\text{vel}}$ and $e_{\text{pre}}$ in different time levels with different $H$ while fixing $J = 5$ and $\ell = 3$. One can see that the error (either in pressure or velocity) decreases as the coarse mesh size decreases. 

\begin{table}[h!]
	\centering
	\begin{tabular}{c|c|c||c|c|c|c}
		$J$ & $\ell$ & $H$ & $t = 0.1$ & $t = 0.2$ & $t = 0.3$ & $t = 0.4$\\ 
		\hline 
		$5$ & $3$& $\sqrt{2}/10$ & $5.2681\%$ & $5.5469\%$ & $6.3184\%$ & $6.7059\%$ \\ 
		\hline
		$5$ & $3$& $\sqrt{2}/20$ & $2.2552\%$ & $2.6246\%$ & $2.8948\%$ & $2.9263\%$ \\ 
		\hline
		$5$ & $3$& $\sqrt{2}/40$ & $1.0362\%$ & $1.0371\%$ & $1.1690\%$ & $1.2529\%$ \\ 
	\end{tabular}
	\caption{$e_{\text{pre}}$ in Example \ref{exp:mw1} with varying coarse mesh size $H$.}
	\label{tab:pe_exp1}
\end{table}

\begin{table}[h!]
	\centering
	\begin{tabular}{c|c|c||c|c|c|c}
	$J$ & $\ell$ & $H$ & $t = 0.1$ & $t = 0.2$ & $t = 0.3$ & $t = 0.4$\\ 
	\hline
	$5$ & $3$ & $\sqrt{2}/10$ & $3.8047\%$ & $4.0189\%$ & $4.3737\%$ & $4.6163\%$ \\ 
	\hline
	$5$ & $3$ & $\sqrt{2}/20$ & $2.4546\%$ & $1.6313\%$ & $1.4103\%$ & $1.4720\%$ \\ 
	\hline
	$5$ & $3$ & $\sqrt{2}/40$ & $1.1189\%$ & $0.6886\%$ & $0.6234\%$ & $0.6390\%$ \\ 
	\end{tabular}
	\caption{$e_{\text{vel}}$ in Example \ref{exp:mw1} with varying coarse mesh size $H$.}
	\label{tab:ve_exp1}
\end{table}

Next, we fix the coarse mesh size $H = \sqrt{2}/20$ and the number of oversampling layers $\ell = 4$. We adjust the number of basis functions $J$ to see how this factor affects the errors in velocity and pressure. Tables \ref{tab:nbp_exp1} and \ref{tab:nbv_exp1} show the errors $e_{\text{vel}}$ and $e_{\text{pre}}$. Under this setting, one may observe that the errors $e_{\text{pre}}$ and $e_{\text{vel}}$ are reduced as the number of basis functions $J$ increases. We remark that once the number of basis functions reaches a certain level, the decay of error becomes slower due to the fact that the decay of eigenvalues obtained in \eqref{eqn:sp_1}-\eqref{eqn:sp_2} slow down. 
\begin{table}[h!]
\centering
	\begin{tabular}{c|c|c||c|c|c|c}
	$J$ & $\ell$ & $H$ & $t = 0.1$ & $t = 0.2$ & $t = 0.3$ & $t = 0.4$ \\ 
	\hline
	$1$ & $4$ & $\sqrt{2}/20$ & $9.6344\%$ & $9.8757\%$ & $10.5834\%$ & $10.6780\%$ \\ 
	\hline
	$2$ & $4$ & $\sqrt{2}/20$ & $7.9885\%$ & $6.2347\%$ & $6.3412\%$ & $6.4201\%$ \\ 
	\hline
	$3$ & $4$ & $\sqrt{2}/20$ & $3.6292\%$ & $3.0248\%$ & $2.8402\%$ & $2.5880\%$ \\ 
	\hline
	$4$ & $4$ & $\sqrt{2}/20$ & $3.3369\%$ & $2.6432\%$ & $2.4863\%$ & $2.2249\%$ \\ 
	\end{tabular}
	\caption{$e_{\text{pre}}$ in Example \ref{exp:mw1} with varying number of basis functions $J$.}	
	\label{tab:nbp_exp1}
\end{table}

\begin{table}[h!]
\centering
	\begin{tabular}{c|c|c||c|c|c|c}
	$J$ & $\ell$ & $H$ & $t = 0.1$ & $t = 0.2$ & $t = 0.3$ & $t = 0.4$ \\ 
	\hline  	
	$1$ & $4$ & $\sqrt{2}/20$ & $11.7671\%$ & $5.1808\%$ & $4.3885\%$ & $4.6301\%$ \\ 
	\hline
	$2$ & $4$ & $\sqrt{2}/20$ & $7.5439\%$ & $2.7621\%$ & $2.0997\%$ & $2.0953\%$ \\ 
	\hline
	$3$ & $4$ & $\sqrt{2}/20$ & $2.8664\%$ & $1.2072\%$ & $0.9478\%$ & $0.9089\%$ \\ 
	\hline
	$4$ & $4$ & $\sqrt{2}/20$ & $2.0116\%$ & $0.8898\%$ & $0.7309\%$ & $0.7107\%$ \\ 
	\end{tabular}
	\caption{$e_{\text{vel}}$ in Example \ref{exp:mw1} with varying number of basis functions $J$.}	
	\label{tab:nbv_exp1}
\end{table}

Further, we change the number of oversampling layers $\ell$ with fixed coarse mesh size $H = \sqrt{2}/20$ and $J = 6$ to see the relation between the error and the number of oversampling layers. Tables \ref{tab:olp_exp1} and \ref{tab:olv_exp1} show the corresponding results. One may observe that the accuracy of the solution is improved if more layers are included in the simulation. Once the number of layers $\ell$ exceeds a certain level, the decay of error stagnates. We remark that based on the theoretical findings in \cite{chung2018cemmixed}, the number of layers should depend on the logarithm of the maximum value of the contrast. 

\begin{table}[h!]
\centering
	\begin{tabular}{c|c|c||c|c|c|c}
	$J$ & $\ell$ & $H$ & $t = 0.1$ & $t = 0.2$ & $t = 0.3$ & $t = 0.4$ \\ 
	\hline
	$6$ & $1$ & $\sqrt{2}/20$ & $2.9318\%$ & $2.1629\%$ & $2.1566\%$ & $1.8118\%$ \\ 
	\hline	
	$6$ & $2$ & $\sqrt{2}/20$ & $2.6116\%$ & $2.2110\%$ & $2.0092\%$ & $1.7443\%$ \\ 
	\hline	
	$6$ & $3$ & $\sqrt{2}/20$ & $2.6751\%$ & $2.0821\%$ & $2.0062\%$ & $1.8385\%$ \\ 
	\hline	
	$6$ & $4$ & $\sqrt{2}/20$ & $2.5124\%$ & $2.0712\%$ & $2.0029\%$ & $1.8621\%$ \\ 
	\end{tabular}
	\caption{$e_{\text{pre}}$ in Example \ref{exp:mw1} with varying oversampling layers $\ell$.}
	\label{tab:olp_exp1}
\end{table}
\vspace{-0.5cm}
\begin{table}[h!]
\centering
	\begin{tabular}{c|c|c||c|c|c|c}
	$J$ & $\ell$ & $H$ & $t = 0.1$ & $t = 0.2$ & $t = 0.3$ & $t = 0.4$ \\ 
	\hline
	$6$ & $1$ & $\sqrt{2}/20$ & $4.3907\%$ & $5.0154\%$ & $6.2635\%$ & $7.7753\%$ \\ 
	\hline	
	$6$ & $2$ & $\sqrt{2}/20$ & $1.9753\%$ & $0.9736\%$ & $0.8820\%$ & $0.9278\%$ \\ 
	\hline	
	$6$ & $3$ & $\sqrt{2}/20$ & $1.9382\%$ & $0.8356\%$ & $0.6610\%$ & $0.6385\%$ \\ 
	\hline	
	$6$ & $4$ & $\sqrt{2}/20$ & $1.9387\%$ & $0.8224\%$ & $0.6532\%$ & $0.6311\%$ \\ 
	\end{tabular}
	\caption{$e_{\text{vel}}$ in Example \ref{exp:mw1} with varying oversampling layers $\ell$.}
	\label{tab:olv_exp1}	
\end{table}
\end{example}

\begin{example}[Marmousi model] \label{exp:mw2}
In this example, we test the proposed method on the {\it Marmousi} benchmark model. The permeability field used in this example is sketched in Figure \ref{fig:exp_2_kappa}. The source function $f$ is chosen as the first derivative of the Gaussian wavelet with central frequency $f_0>0$
$$ f(x,t) := g(x)(t-2f_0^{-1}) \exp\big(-\pi^2 f_0^2 (t-2f_0^{-1})^2 \big),$$
$$g(x) := 10\delta^{-2} \exp\big(\abs{x-\mathbf{c}}^2\delta^{-2}\big),$$
for $x \in \Omega$, $t \in (0,T]$, and $\mathbf{c} = (0.5, 0.5)^T$. Note that $\delta>0$ measures the size of the support of the source. Here, we denote the two-dimensional Euclidean distance as $\abs{x-y}$ for any $x, y \in \mathbb{R}^2$. 

The profiles of solutions at the terminal time $t = T$ are sketched in Figures \ref{fig:ref_exp2} and \ref{fig:ms_exp2}. One may see that the proposed multiscale method can capture most of the details of the reference solution with less computational cost. 

    \begin{figure}[h!]
        \centering
        \mbox{
        \subfigure[Pressure $p_h$.]{
        \includegraphics[width=2.7in]{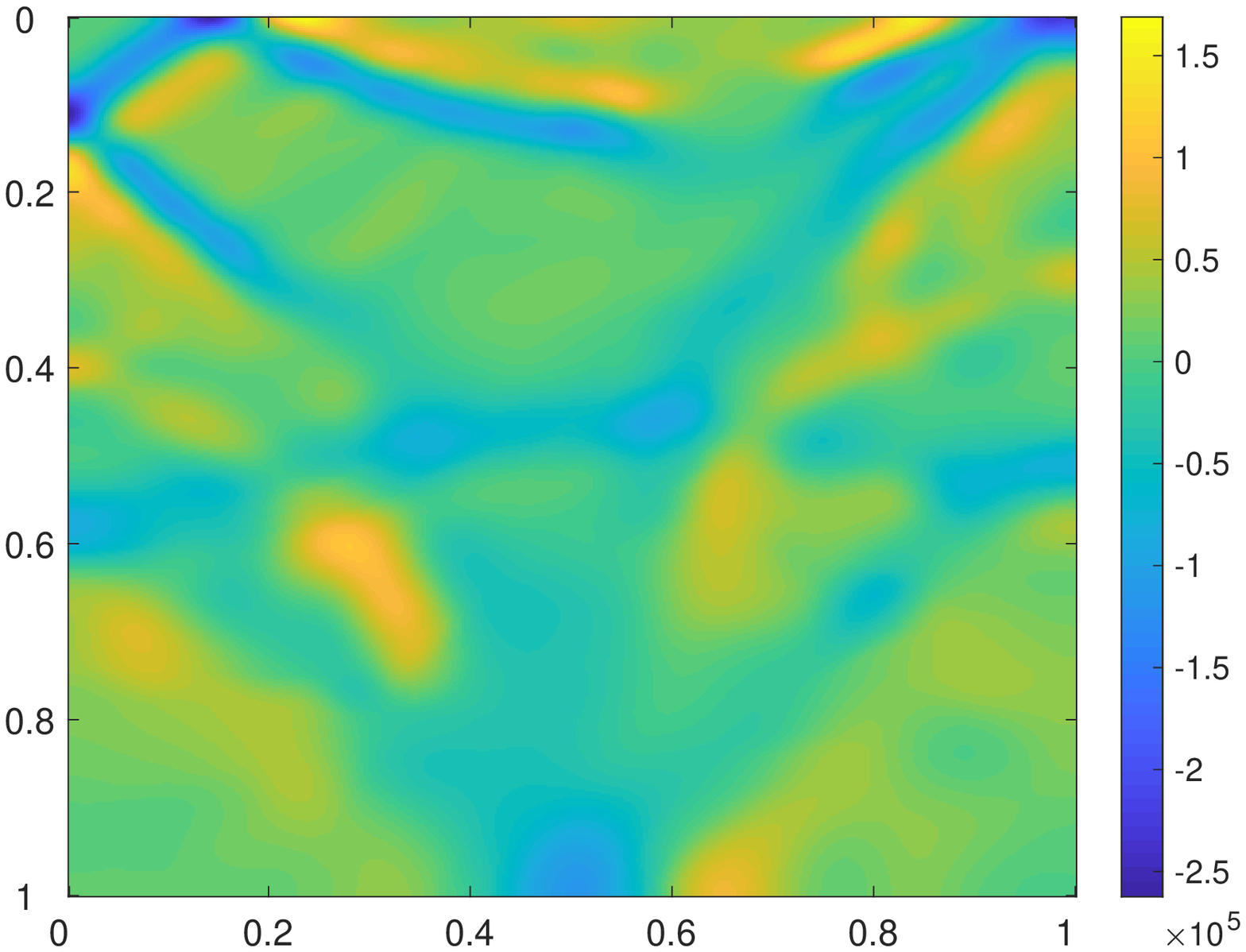}
        \label{fig:ref_p_exp2}}
        \subfigure[Velocity $v_h$.]{
        \includegraphics[width=2.7in]{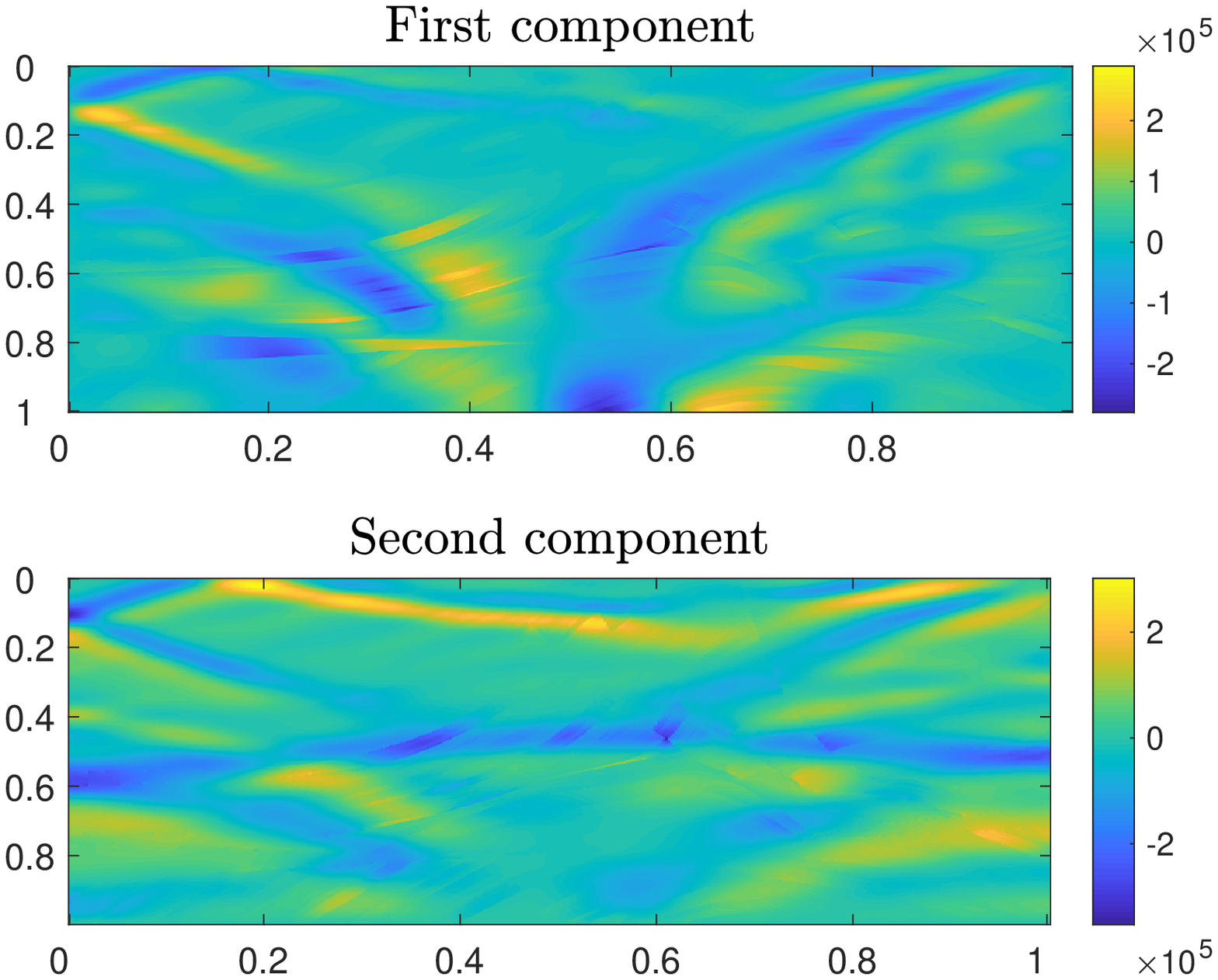}
        \label{fig:ref_v_exp2}}
        }
        \caption{Reference solution at $t = T$ in Example \ref{exp:mw2}. $f_0 = 20$; $\delta = 0.02$. }
       \label{fig:ref_exp2}
    \end{figure}

    \begin{figure}[h!]
        \centering
        \mbox{
        \subfigure[Pressure $p_{\text{ms}}$.]{
        \includegraphics[width=2.7in]{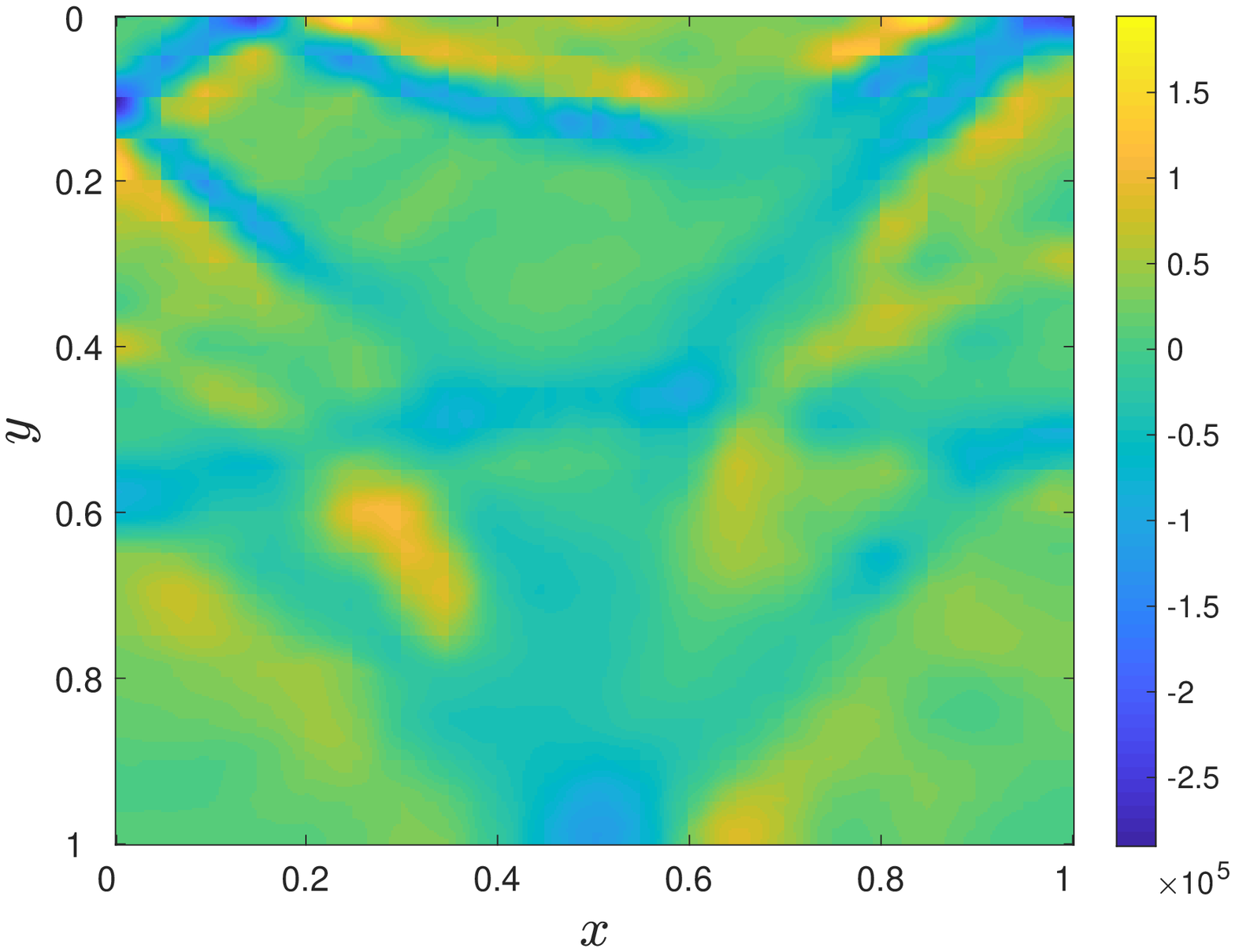}
        \label{fig:ms_p_exp2}}
        \subfigure[Velocity $v_{\text{ms}}$.]{
        \includegraphics[width=2.7in]{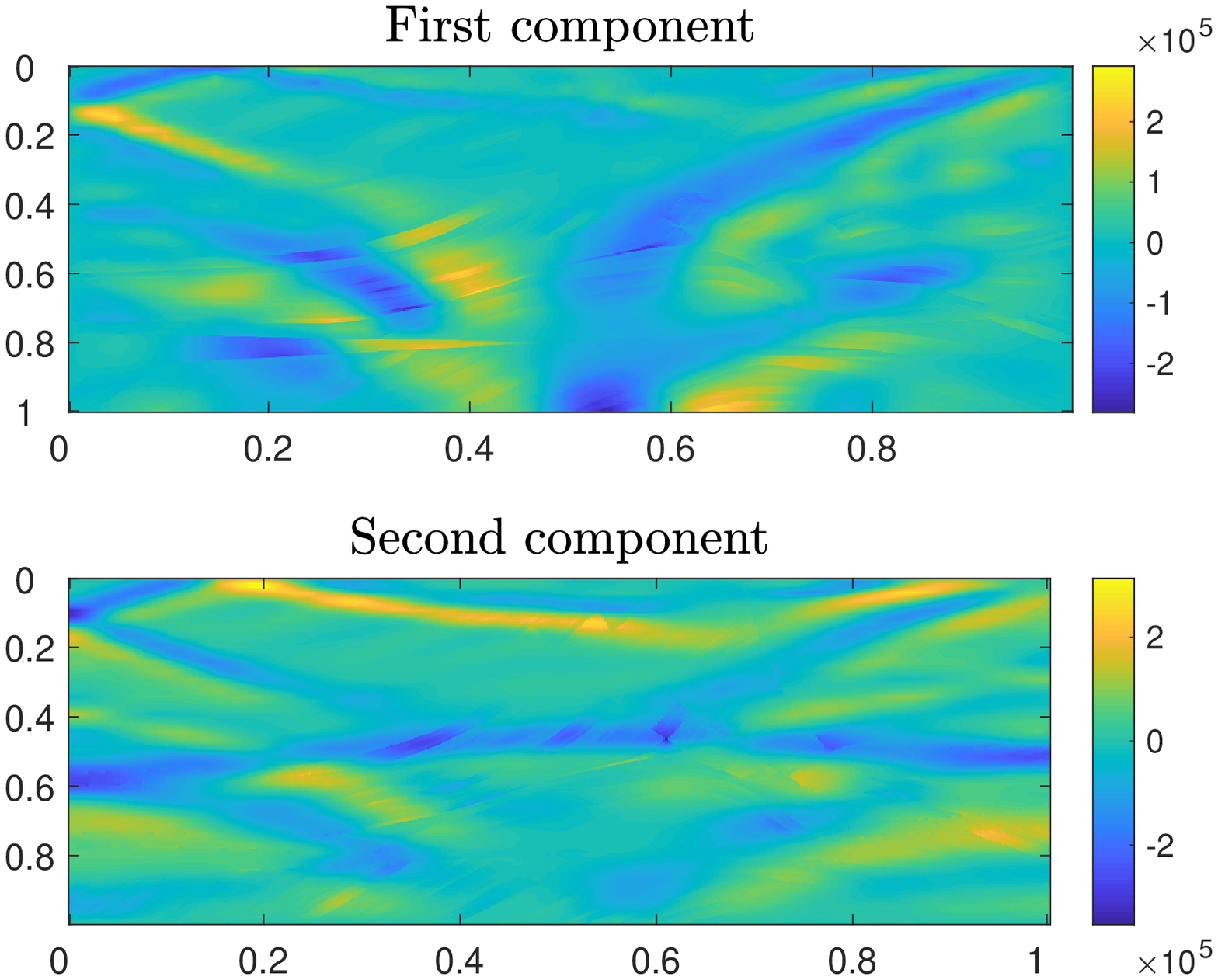}}
	\label{fig:ms_v_exp2}
        }
        \caption{Multiscale approximation at $t = T$ in Example \ref{exp:mw2}. $f_0 = 20$; $\delta = 0.02$; $H = \sqrt{2}/20$; $J = 5$; $\ell = 3$.}
        \label{fig:ms_exp2}
    \end{figure}
    
We report the results of errors in both velocity and pressure using the proposed multiscale method with varying coarse mesh size. Tables \ref{tab:pe_exp2} and \ref{tab:ve_exp2} record the errors with fixed $J = 5$, $\ell = 3$, $f_0 = 20$, and $\delta = 0.02$. We remark that $\delta = 4\sqrt{2} h$ in this case. 
One can observe that the errors in pressure and velocity reduce as the coarse mesh size decreases. 
\begin{table}[h!]
\centering
	\begin{tabular}{c|c|c|c||c|c|c|c}
	$J$ & $\ell$ & $(f_0, \delta)$ & $H$ & $t = 0.1$ & $t = 0.2$ & $t = 0.3$ & $t = 0.4$ \\ 
	\hline
	$5$& $3$ & $(20,0.02)$ & $\sqrt{2}/10$ & $12.6666\%$ & $49.7508\%$ & $66.1003\%$ & $73.1510\%$ \\ 
	\hline	
	$5$& $3$ & $(20,0.02)$ & $\sqrt{2}/20$ & $7.3507\%$ & $13.3181\%$ & $24.2448\%$ & $33.1721\%$ \\ 
	\hline	
	$5$& $3$ & $(20,0.02)$ & $\sqrt{2}/40$ & $2.6322\%$ & $3.3151\%$ & $4.7303\%$ & $6.3163\%$ \\ 
	\hline	
	$5$& $3$ & $(20,0.02)$ & $\sqrt{2}/80$ & $0.9754\%$ & $1.0834\%$ & $1.2581\%$ & $1.3355\%$ \\ 
	\end{tabular}
	\caption{$e_{\text{pre}}$ in Example \ref{exp:mw2} with varying coarse mesh size $H$.}
	\label{tab:pe_exp2}
\end{table}

\begin{table}[h!]
\centering
	\begin{tabular}{c|c|c|c||c|c|c|c}
	$J$ & $\ell$ & $(f_0, \delta)$ & $H$ & $t = 0.1$ & $t = 0.2$ & $t = 0.3$ & $t = 0.4$ \\ 
	\hline
	$5$& $3$ & $(20,0.02)$ & $\sqrt{2}/10$ & $42.6513\%$ & $46.1319\%$ & $71.7437\%$ & $88.0726\%$ \\ 
	\hline
	$5$& $3$ & $(20,0.02)$ & $\sqrt{2}/20$ & $8.8847\%$ & $10.3469\%$ & $25.1502\%$ & $33.1973\%$ \\ 
	\hline
	$5$& $3$ & $(20,0.02)$ & $\sqrt{2}/40$ & $1.8760\%$ & $2.4122\%$ & $3.5537\%$ & $5.5809\%$ \\ 
	\hline
	$5$& $3$ & $(20,0.02)$ & $\sqrt{2}/80$ & $0.5822\%$ & $1.1813\%$ & $1.1804\%$ & $1.2015\%$ \\ 
	\end{tabular}
	\caption{$e_{\text{vel}}$ in Example \ref{exp:mw2} with varying coarse mesh size $H$.}	
	\label{tab:ve_exp2}
\end{table}

Furthermore, we calculate the multiscale solution by the proposed method with different $f_0$ and $\delta$ when $H = \sqrt{2}/20$, $J = 5$, and $\ell = 3$. Tables \ref{tab:pe_exp2_cfdelta} and \ref{tab:ve_exp2_cfdelta} present the corresponding numerical results. We remark that the for high frequency cases with larger $f_0$, one may use a finer coarse mesh to enhance the accuracy of the multiscale approximation. 

\begin{table}[h!]
\centering
	\begin{tabular}{c|c|c|c||c|c|c|c}
	$J$ & $\ell$ & $(f_0, \delta)$ & $H$ & $t = 0.1$ & $t = 0.2$ & $t = 0.3$ & $t = 0.4$ \\ 
	\hline
	$5$& $3$ & $(20,0.02)$   & $\sqrt{2}/20$ & $7.3507\%$ & $13.3181\%$ & $24.2448\%$ & $33.1721\%$ \\ 
	\hline	
	$5$& $3$ & $(20,0.005)$ & $\sqrt{2}/20$ & $8.9192\%$ & $16.8415\%$ & $29.8360\%$ & $39.8137\%$ \\
	\hline	
	$5$& $3$ & $(50,0.02)$   & $\sqrt{2}/20$ & $69.5498\%$ & $92.0786\%$ & $101.8084\%$ & $107.6769\%$ \\ 
	\hline	
	$5$& $3$ & $(50,0.005)$ & $\sqrt{2}/20$ & $43.1308\%$ & $68.3501\%$ & $83.9901\%$ & $89.6999\%$ \\ 
	\end{tabular}
	\caption{$e_{\text{pre}}$ in Example \ref{exp:mw2} with varying $f_0$ and $\delta$.}
	\label{tab:pe_exp2_cfdelta}
\end{table}

\begin{table}[h!]
\centering
	\begin{tabular}{c|c|c|c||c|c|c|c}
	$J$ & $\ell$ & $(f_0, \delta)$ & $H$ & $t = 0.1$ & $t = 0.2$ & $t = 0.3$ & $t = 0.4$ \\ 
	\hline
	$5$& $3$ & $(20,0.02)$   & $\sqrt{2}/20$ & $8.8847\%$ & $10.3469\%$ & $25.1502\%$ & $33.1973\%$ \\ 
	\hline
	$5$& $3$ & $(20,0.005)$ & $\sqrt{2}/20$ & $49.8404\%$ & $13.8201\%$ & $30.8732\%$ & $40.2296\%$ \\ 
	\hline
	$5$& $3$ & $(50,0.02)$   & $\sqrt{2}/20$ & $66.3617\%$ & $92.1064\%$ & $99.8070\%$ & $109.3254\%$ \\ 
	\hline
	$5$& $3$ & $(50,0.005)$ & $\sqrt{2}/20$ & $38.7407\%$ & $68.9757\%$ & $79.5035\%$ & $92.6214\%$ \\ 
	\end{tabular}
	\caption{$e_{\text{vel}}$ in Example \ref{exp:mw2} with varying $f_0$ and $\delta$.}
	\label{tab:ve_exp2_cfdelta}
\end{table}

\end{example}

\section{Conclusion}\label{sec:conclusion}
In this work, we have proposed and analyzed the constraint energy minimizing generalized multiscale finite element method for solving the wave equation in mixed formulation. The multiscale basis functions for pressure are obtained by solving a class of well-designed local spectral problems. Based on the concept of constraint energy minimization, we construct the multiscale basis functions for velocity satisfying the property of least energy. The method is shown to have first-order convergence with respect to the coarse mesh size. Numerical results are provided to illustrate the efficiency of the proposed method. 

\subsection*{Acknowledgement}
Eric Chung's work is partially supported by Hong Kong RGC General Research Fund (Projects 14304217 and 14302018) and CUHK Direct Grant for Research 2018-19.

\bibliographystyle{abbrv}
\bibliography{references}
\end{document}